\documentclass[12pt,reqno]{article}

\usepackage[usenames]{color}
\usepackage{amssymb}
\usepackage{graphicx}
\usepackage{amscd}

\usepackage[colorlinks=true,
linkcolor=webgreen,
filecolor=webbrown,
citecolor=webgreen]{hyperref}

\definecolor{webgreen}{rgb}{0,.5,0}
\definecolor{webbrown}{rgb}{.6,0,0}

\usepackage{color}
\usepackage{fullpage}
\usepackage{float}

\usepackage{graphics,amsmath,amssymb}
\usepackage{amsthm}
\usepackage{amsfonts}
\usepackage{latexsym}

\begin{document}

\vspace*{2.1cm}

\theoremstyle{plain}
\newtheorem{theorem}{Theorem}
\newtheorem{corollary}[theorem]{Corollary}
\newtheorem{lemma}[theorem]{Lemma}
\newtheorem{proposition}[theorem]{Proposition}
\newtheorem{obs}[theorem]{Observation}
\newtheorem{claim}[theorem]{Claim}

\theoremstyle{definition}
\newtheorem{definition}[theorem]{Definition}
\newtheorem{example}[theorem]{Example}
\newtheorem{remark}[theorem]{Remark}
\newtheorem{conjecture}[theorem]{Conjecture}

\begin{center}


{\Large\bf List and total colorings of multiset permutation graphs}
 
\vskip 1cm
\large
Italo J. Dejter

University of Puerto Rico

Rio Piedras, PR 00936-8377

\href{mailto:italo.dejter@gmail.com}{\tt italo.dejter@gmail.com}
\end{center}

\begin{abstract}
Let $k$ and $\ell$ be positive integers. The multiset star transposition graph ST$_k^\ell$ has as vertices the $k\ell$-strings $v_0\cdots v_{k\ell-1}$ on $k$ symbols, each symbol repeated $\ell$ times, and edges given by the transpositions $(v_0\;v_i)$ with $v_i\ne v_0$ ($0<i<k\ell$).
It is shown for $k>1$ and $\ell>2$ that ST$_k^\ell$ is $(\ell-1)$-choosable and that, as a result, admits total colorings. In order to prove such assertions,
the notion of efficient domination set (or E-set) of a graph is generalized for $\ell>1$ to that of an efficient dominating$\,^\ell$-set and applied to the graphs ST$_k^\ell$\,, showing they admit vertex partitions 
that generalize the Dejter-Serra partitions of ST$_k^1$ into E-sets, but not efficiently in the sense that the distance of each E$^\ell$-set be 3.
Efficiently in such sense however, $ST^2_k$ and the related 2-set pancake permutation graph PC$^2_k$, among other intermediate permutation graphs, are shown to admit total colorings with $2k-1$ colors that determine partitions into $2k-1$ E-sets, each with distance 3. 
 Furthermore, associated E-chains are examined.
 \end{abstract}

\section{Introduction}\label{s1}

Let $0<k,\ell\in\mathbb{Z}$. Given a finite graph $G=(V(G),E(G))$ of girth larger than 3 and a subset $S\subseteq V(G)$, we say that $S$ is an {\it efficient dominating$\,^\ell$-set} (or {\it E$\,^\ell$-set}) or a {\it perfect$\,^\ell\!$code}, if for each $v\in V(G)\setminus S$ there exist exactly $\ell$ vertices $v^0,v^1,\ldots,v^{\ell-1}$ in $S$ such that $v$ is adjacent to  $v^i$, for every $i\in[\ell]=\{0,\ldots,\ell-1\}$ and in addition, if $\ell>1$, then $v$ is the only vertex of $G$ in the intersection of the 1-spheres of  $v^0,v^1,\ldots,v^{\ell-1}$. Since the girth of $G$ is required to be larger than 3, the {\it dominating$\,^\ell$-set} (or {\it D$\,^\ell$-set}) $S(v)=\{v^0,v^1,\ldots,v^{\ell-1}\}$ of $v$ {\it with respect to} (or {\it wrt}) $S$ is an independent subset of $G$. 
Clearly, an E$\,^\ell$-set $S$ of $G$ is the union of the D$\,^\ell$-sets wrt $S$.
In particular, 
 if $\ell=1$ then $S$ is a {\it 1-perfect code} (or {\it 1-error correcting code}) \cite{Borges,D80,Huffman,Vera,MWS}, also said to be an {\it efficient dominating set} (or {\it E-set}) \cite{AK,BBS,gen,D73,D76,edig}. 
 
\begin{example} Here is an example of what is {\it not} an efficient dominating set:
Let $X=\{w_0,w_1\}$ and $Y=\{v,v',v''\}$ be the vertex parts of the bipartite graph $G=K_{2,3}$. Then, the stable set $S(v)=X$ is {\it not} a 2-efficient dominating set of $G$ because  the intersection of the 1-spheres of $w_0$ and $w_1$ is all of $Y$, not a sole vertex, as required in the definition above. 
\end{example}

Combined with total coloring, (whose definition is recalled shortly below),  E$\,^\ell$-sets are addressed in Sections~\ref{s2}--\ref{Galphak} and its sub-case $\ell=2$ in Sections~\ref{list}--\ref{panqueque}, in extending results of \cite{D73} and \cite{cong}. The present treatment considers the {\it multiset star transposition graphs} introduced in Section~\ref{s2} and their pancake permutation graph variation in Section~\ref{panqueque}, among others. In Sections~\ref{list}-\ref{nuevo}, these multiset cases are shown to lead to list-coloring \cite{Erdos,toft,V2} applications. Other  
applications of E$\,^\ell$-sets occur in the Theory of Error-Correcting Codes \cite{Huffman,MWS,Vera}. 

A {\it total coloring} (or {\it TC}) of a graph $G$ is an assignment of colors to the vertices and edges of $G$ such that no two incident or adjacent elements (vertices or edges)
are assigned the same color \cite{tc-as}.
The {\it TC Conjecture}, posed independently by Behzad \cite{B1,B2} and Vizing \cite{V},  asserts that the {\it total chromatic number} $\chi''(G)$ of a graph $G$ (namely, the least number of colors required by a TC of $G$) is either $\Delta(G)+1$ or $\Delta(G)+2$, where $\Delta$ is the largest degree of any vertex of $G$. A recent survey \cite{tc-as} contains an updated bibliography on TCs following the previous account of~\cite{Yap}.   
 The TC Conjecture was established for cubic graphs \cite{Feng,Mazzu,Rosen,Vi}, meaning that the total chromatic number of cubic graphs is either 4 or 5. To decide whether a cubic graph $G$ has total chromatic number $\Delta(G)+1$, even for bipartite cubic graphs, is NP-hard \cite{Arroyo}.

A total coloring of a $k$-regular graph $G$ such that the vertices adjacent to each $v\in V(G)$ together with $v$ itself are assigned pairwise different colors of the color set $[k]=\{0,1,\ldots,k-1\}$ will be said to be an {\it efficient coloring} if each $v\in V(G)$ together with its neighbors are assigned {\it all} the colors of $[k]$. 
However, this is not a property suitable for E$^\ell$-sets, as the minimum distance of each E$^\ell$-set is 2, not 3.
The study of efficient colorings was considered in~\cite{prev} and~\cite{cong} for $k$-regular graphs with $k+1$ colors.

\section{Multiset star transposition graphs}\label{s2}

Let $0<\ell\in\mathbb{Z}$ and let $1<k\in\mathbb{Z}$. A string over the alphabet $[k]$ that contains exactly $\ell$ occurrences of each $i\in[k]$ is said to be an $\ell$-{\it set permutation}.
In denoting specific $\ell$-set permutations, commas and brackets are often omitted. 

Let $V^\ell_k$ be the set of all $\ell$-set permutations of length $k\ell$. Let the {\it star $\ell$-set transposition graph} $ST^\ell_k$ be the graph on vertex set $V^\ell_k$ with an edge between each two vertices $v=v_0v_1\cdots v_{k\ell-1}$ and $w=w_0w_1\cdots w_{k\ell-1}$ that differ in a {\it star transposition}, i.e. by swapping the first entry $v_0$ of $v=v_0v_1\cdots v_{k\ell-1}\in V^\ell_k$ with any entry $v_j$, ($j\in[k\ell]\setminus\{0\}$), whose value differs from that of  $v_0$, (so $v_j\ne v_0$), thus obtaining either $$w=w_0\cdots w_j\cdots w_{k\ell-1}=v_j\cdots v_0\cdots w_{k\ell-1}\;\mbox{ or }\;w=w_0\cdots w_{k\ell-1}=v_{k\ell-1}\cdots v_0,$$
where all not mentioned entries (represented in the ellipses) remain unchanged.  In other words,
each edge of $ST^\ell_k$ is given by the transposition of the initial entry of an endvertex string with an entry that contains a different symbol than that of the initial entry. The graphs $ST_k^\ell$ are a particular case of the graphs treated in \cite{faltaba} in a context of determination of Hamilton cycles.

It is known that all $k$-permutations, (that is: all 1-set permutations of length $k$), form the {\it symmetric group}, denoted $Sym_k$, under composition of
$k$-permutations, each $k$-permutation $v_0v_1\cdots v_{k-1}$ taken as a bijection from the {\it identity} $k$-permutation $01\cdots(k-1)$ onto $v_0v_1\cdot v_{k-1}$ itself. 
A graph $ST^1_k$ with $k>1$ (which excludes $ST^1_1$) is the Cayley graph of $Sym_k$ with respect to the set of transpositions $\{(0\;i); i\in[k]\setminus\{0\}\}$.
Such a graph is denoted $ST_k$ in \cite{AK, D73}, where is proven that its vertex set admits a partition into $k$ E-sets, exemplified on the upper left of Figure~\ref{fig1}
for $ST^1_3=ST_3$, with the vertex parts of the partition differentially colored in blue, red and green, for first entries 0, 1 and 2, respectively. Figure 1 of \cite{D73} shows a similar example for $ST^1_4=ST_4$.
Also, the graph $ST^2_k$ is vertex transitive, but is not a Cayley graph; it is neither a Schreier graph with respect to the subgroup specified in Remark~\ref{Sch}; (even though every regular graph can be presented as a Schreier graph).

Let $k>1$. 
We consider in Section~\ref{Galphak},
for each $i\in[k]$, among the vertices $v=v_0\dots v_{k\ell -1}$ of $ST^\ell_k$, those that have the first entry $v_0$ equal to $i$ and find they form an E$\,^\ell$-set $S=S_i^k$ of $ST^\ell_k$, to be said to be an {\it SE$\,^\ell$-set}. In addition, we consider in Sections~\ref{s3}-\ref{panqueque} the sets $\Sigma_i^k$ of vertices $v_0v_1\cdots v_{2k-1}$ of $ST^2_k$ such that $v_0=v_i$ for exactly one value $i=1,\ldots,2k-1$, as there are just two occurrences of each symbol in every vertex of $ST^\ell_k$. 
In preparation, Sections~\ref{list}-\ref{nuevo} shows, for $k>1$, that
\begin{enumerate}\item if $\ell>1$, then $ST_k^\ell$ has a total coloring;
\item if $\ell>2$, then $ST_k^\ell$ is $(\ell-1)$-choosable; 
\item $ST_k^2$ has an efficient coloring via the sets $\Sigma_i^k$, to be referred to as {\it $\Sigma$E-sets}.
\end{enumerate} 
 Theorem~\ref{chi} (Section~\ref{prelim}) is a structural result for efficient colorings of $ST_k^2$ as in item 3.

Let $\ell\in\{0,1\}$. In Subsection~\ref{Cs} we generalize, by means of $\ell$-set permutations (see Section~\ref{s2}), the result of \cite{D73} that the star transposition graphs form a {\it dense segmental neighborly E-chain}. In Section~\ref{panqueque} we generalize multiset star transposition graphs to multiset pancake permutation graphs and related intermediate graphs \cite{D73} leading to a suitable version of dense neighborly E-chain with obstructions preventing segmental E-chains. 

\section{SE$\,^\ell$-sets of $\ell$-set star transposition graphs}
\label{Galphak}

The vertices of $ST^\ell_k$ are the multiset permutations $v_0\dots v_{k\ell -1}$ of the string $$\overbrace{0\cdots 0}\overbrace{1\cdots 1}\overbrace{2\cdots 2}\cdots\overbrace{(k-1)\cdots(k-1)}=0^\ell  1^\ell 2^\ell \cdots(k-1)^\ell .$$         
For each $i\in[k]$, among the vertices $v=v_0\dots v_{k\ell -1}$ of $ST^\ell_k$, let us show that those that have the first entry $v_0$ equal to $i$ constitute an SE$\,^\ell$-set $S=S_i^k$ of $ST^\ell_k$.   

\begin{theorem}\label{teo1}\label{t} Let $0<\ell ,k\in\mathbb{Z}$ and let $ST^\ell_k\ne ST^1_2$.
For each $i\in[k]$, the multiset permutations $v_0\dots v_{k\ell -1}$ of the string 
$0^\ell  1^\ell 2^\ell \cdots(k-1)^\ell$ that have their first entry $v_0$ equal to $i$ constitute 
an SE$\,^\ell$-set $S_i^k$ of the graph $ST^\ell_k$. Each D$\,^\ell$-set wrt $S_i^k$ is an independent set of $ST^\ell_k$. 
\end{theorem}

\begin{proof}
For fixed $i\in[k]$, each vertex $v=v_0v_1\cdots,v_{k\ell -1}$ of $E(ST^\ell_k)\setminus S_i^k$ has initial entry $v_0=j$, for some $j\in[k]$ such that $j\ne i$. Then, $v$ is adjacent to $\ell$ vertices of $S_i^k$ obtained by transposing the position of each of the $\ell$ entries $v_h=i$ with the position of that initial entry $v_0=j$, where $h\in\{1,\ldots,k\ell -1\}$. The graph induced by the edges of such adjacencies form a copy $H$ of the complete bipartite graph $K_{1,\ell }$ with $v$ as its sole degree-$\ell$ vertex and its leaves (if $H$ is taken as a rooted tree with $v$ as its root) as the vertices $v^j$ in $S_i^k$
obtained from $v$ by transposing $v_0=j$ with the entries $v_h^j=i$ of those vertices. Since $ST^\ell_k\ne ST^1_2$ has girth larger than 3, then each D$\,^\ell$-set wrt $S_i^k$ is an independent set of $ST^\ell_k$. 
\end{proof}

\begin{corollary}
 The vertex set $V(ST^\ell_k)$ admits a partition into $k$ SE$\,^\ell$-sets $S_i^k$, where $i\in[k]$. 
\end{corollary}

\begin{proof} The SE$\,^\ell$-sets $S_i^k$ form a partition of $V(ST^\ell_k)$, since each such $S_i^k$ is composed precisely by the vertices $v=v_0v_1\cdots v_{k\ell -1}$ of $ST^\ell_k$ having initial entry $v_0=i$, which form precisely one of the $k$ parts of the partition.
\end{proof}

Let $2ST^\ell_k$ be the multigraph obtained from $ST^\ell_k$ by replacing each edge $e$ of $ST^\ell_k$ by two parallel edges with the same endvertices of $e$. 

\begin{corollary}
Let $i\in[k]$. Each vertex of $S_i^k$ belongs to $k\ell -1$ D$\,^\ell$-sets wrt $S_i^k$, where the induced graph of each D$\,^\ell$-set is isomorphic to the complete bipartite graph $K_{1,\ell }$. 
The set of all such D$\,^\ell$-sets, for every $i\in[k]$, form a partition of the edge set of $2ST^\ell_k$.
\end{corollary}

\begin{figure}[htp]
\includegraphics[scale=1.1]{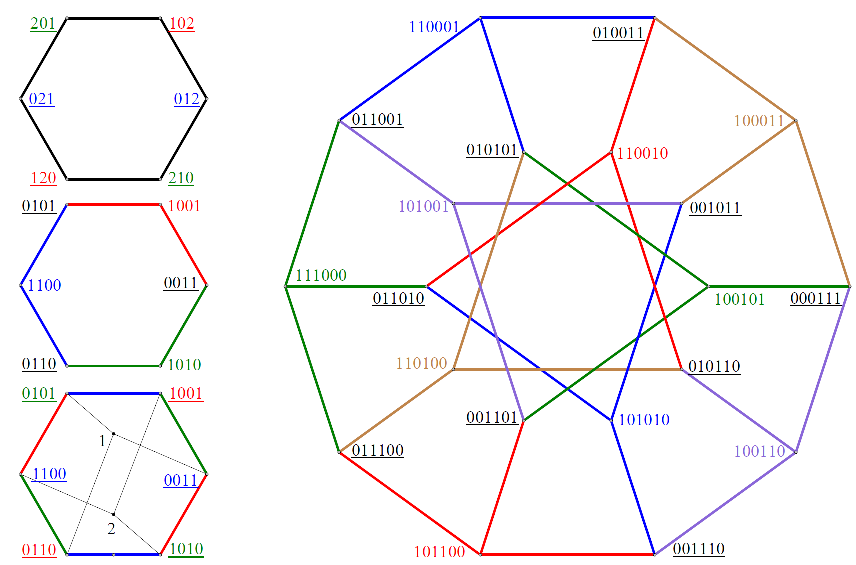}
\caption{Colorings of $ST_3^1=ST_3$, $ST_2^2$ (twice) and the Desargues graph $ST_2^3$.}
\label{fig1}
\end{figure}

\begin{proof}
Since there are $k-1$ values $j\ne i$ in $[k]$, each vertex $v\in S_i^k$ belongs to $k\ell -1$ D$\,^\ell$-sets wrt $S_i^k$. Since each such $v$ is the neighbor of $\ell$ vertices with first entry $v_0=j$, for each $j\in[k]$ with $j\ne i$, then the induced graph of each D$\,^\ell$-set is isomorphic to $K_{1,\ell }$\,,  noted already in the proof of Theorem~\ref{t}. Also, each edge $e$ of $ST^\ell_k$ with endvertices $v$ and $w$ having respective first entries $i$ and $j$ is both a member of $S_i^k$ and of $S_j^k$. Thus, $e$ appears in $2ST^\ell_k$ as two parallel edges $e^i$ and $e^j$ in the respective induced subgraphs of $S_i^k$ and $S_j^k$.
\end{proof}

\begin{corollary}
If $k=2$, each $S_i^k=S_i^2$ as in Theorem~\ref{teo1} induces an edge partition $P(S_i^2)$ of $ST^\ell_2$ whose monochromatic components are copies of $K_{1,\ell }$
in $ST_2^\ell$. 
\end{corollary}

\begin{proof}
 For each fixed $i=0,1$, the copies of $K_{1,\ell }$ induced by the D$\,^\ell$-sets $S_i^2(v)$ wrt $S_i^2$ form a partition of the edge set of $ST^\ell_2$ whose monochromatic components are precisely such copies.  
\end{proof}

\begin{example} Here are two examples of SE$\,^\ell$-sets, for $\ell=2,3$, where elements of SE$\,^\ell$-sets appear underlined in the accompanying Figure~\ref{fig1}.

The graph $ST_2^2$ is the 6-cycle graph $(0011,1001,0101,1100,0110,1010)$, represented in the middle left of Figure~\ref{fig1} and showing in distinct edge shades the induced subgraphs of the composing D$^2$-sets of the SE$\,^2$-set $S=\{0011,0101,0110\}$, namely the D$^2$-set $\{0011,0101\}$ of 1001 wrt $S$, the D$^2$-set $\{0011,0110\}$ of 1010 wrt $S$, and the D$^2$-set $\{0110,0101\}$ of 1100 wrt $S$. 

 The graph $ST_2^3$ is the Desargues graph on the right of Figure~\ref{fig1}, where the subgraphs $K_{1,3}$ induced by the D$^3$-sets $S_0^2(v)$ of the vertices $v=1v_1\cdots v_5$ of $ST_2^3$ wrt $S_0^2$ have the edges in red, blue, green, hazel and violet, in order to differentiate them. These subgraphs $K_{1,3}$ appear in ``opposing'' pairs, allowing to raise the question of how many colors are necessary to color such edge partitions.
\end{example}

\begin{example}
The vertices of $ST_3^2$ are the multiset permutations of $v_0^0=001122$, a total of $|V(ST_3^2)|=\frac{6!}{2!2!2!}=\frac{720}{8}=90$ vertices. The regular degree of $ST_3^2$ is 4.
The graph $ST^3_2$ has the SE$^3$-set $S_0^2=\{v_0\cdots v_5\in V(ST(_3^2);v_0=0\}$. 
For example, $v=100122$ has $S^k_0(v)=S^2_0(v)=\{010122,001122\}$ as its D$\,^\ell$-set wrt $S$, 
and $v'=120120$ has $S^k_0(v')=S^2_0(v')=\{021120,020121\}$ as D$\,^\ell$-set wrt $S_0^2$. Each vertex $v$ in $S_0^2$ belongs to $k\ell-1=4$ D$^2$-sets wrt $S_0^2$. 
While $ST_3^2$ has 90 vertices, $S$ has $\frac{90}{k}=\frac{90}{3}=30$ vertices.
For example, 010122 belongs to $S^2_0(100122)$, $S^2_0(110022)$, $S^2_0(210102)$ and $S^3_0(210120)$. Specifically as in display (\ref{ccc}).
\begin{eqnarray}\label{ccc}\begin{array}{ccc}
S_0^3(100122)=\{010122,001122\},\\
S_0^3(110022)=\{010122,011022\},\\
S_0^3(210102)=\{010122,012102\},\\
S_0^3(210120)=\{010122,012120\}.\end{array}\end{eqnarray}
\end{example}

\section{List coloring of $\ell$-set star transposition graphs}
\label{list}

\begin{theorem}\label{pec} Let $k>0$ and let $\ell>1$. Then, the graph
$ST_k^\ell$ has a proper edge coloring.
\end{theorem}

\begin{figure}[htp]
\includegraphics[scale=1.44]{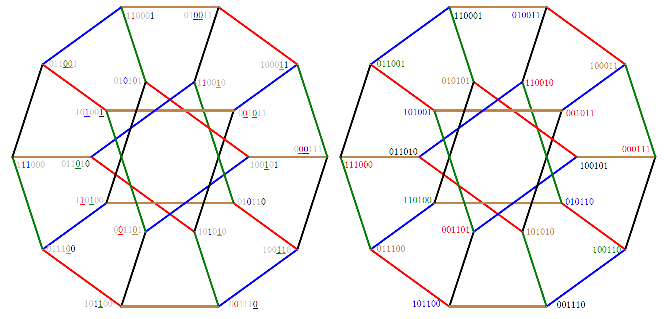}
\caption{On 2-choosability of the Desargues graph $ST_2^3$.}
\label{fig2}
\end{figure}

\begin{proof}
A proper edge coloring of $ST_k^\ell$ is given as follows. Let $i\in[k\ell]\setminus\{0\}$. Then, 
each edge 
\begin{eqnarray}\label{notice}
(v,w)=(v_0v_1\cdots v_i\cdots v_{k\ell-1},\;v_iv_1\cdots v_0\cdots v_{k\ell-1}),
\mbox{ where }v_i\ne v_0,
\mbox{ is assigned color }i.
\end{eqnarray} 
The resulting edge color assignment also is to include that the edge
\begin{eqnarray}\label{not2}\begin{array}{ll}
(v,w)=(v_0v_1\cdots v_{k\ell-1},\;v_1v_0\cdots v_{k\ell-1}),\mbox{ where }v_1\ne v_0,&\mbox{ is assigned color }1,\mbox{ and}\\
(v,w)=(v_0v_1\cdots v_{k\ell-1},\;v_{k\ell-1}v_1\cdots v_{0}),\mbox{ where }v_{2k-1}\ne v_0,&\mbox{ is assigned color }k\ell-1.
\end{array}\end{eqnarray} 
As a result, the edges incident to each vertex of $ST_k^\ell$ have pairwise different colors, so that $ST_k^\ell$ has a proper edge coloring.
\end{proof}

Given a graph $G$ and a color subset (said to be a {\it list}) $L(v)$, $\forall v\in V(G)$, a {\it list coloring} of $G$ is a function $\lambda$ that maps every $v\in V(G)$ to a color $\lambda(v)\in L(v)$ \cite{Erdos,toft,Thomassen,V2}. Such list coloring is {\it proper} if no two adjacent vertices get the same color. A graph is $k$-{\it choosable} if it has a proper list coloring, no matter how one assigns a list of $k$ colors to each vertex.

\begin{theorem}\label{choo}
Let $k>1$ and let $\ell>2$. Then, the graph $ST_k^\ell$ is $(\ell-1)$-choosable. However, no $(\ell-1)$-list coloring of $ST_k^\ell$ contains only efficient color classes.
\end{theorem}

\begin{proof}
There are $\ell-1$ occurrences of the first entry $v_0$ in any vertex $v=v_0\cdots v_{k\ell-1}$ of $ST_k^\ell$. The remaining $(k-1)\ell$ entries in $\{1,\ldots,k\ell-1\}$ determine the edges incident to $v$; these edges are colored according to displays (\ref{notice}) and (\ref{not2}). The said $\ell-1$ occurrences constitute a list $L(v)$, $\forall v\in V(ST_k^\ell)$. Adjacent vertices $v=v_0\cdots v_{k\ell-1}$ and $w=w_0\cdots w_{k\ell-1}$ have $L(v)\cap L(w)=\emptyset$, since $v_0\neq w_0$, where $w_0=v_j$ for some $j\in\{1,\dots,k\ell-1\}$ and $v_j\neq v_0$, that is: $v_j$ must have a different value in $\{0,\ldots,k-1\}$ than that of $v_0$. Thus, $ST_k^\ell$ is $(\ell-1)$-choosable.

To show that no $(\ell-1)$-list coloring of $ST_k^\ell$ contains only efficient color classes, note that each vertex $w$ at distance 2 from say vertex $v=00\cdots 0v_\ell\cdots,v_{k\ell-1}$
starts with $w_0=0$ and only has one entry in $\{w_1,\ldots,w_{\ell-1}\}$ distinct from 0. But there are $(\ell-1)(\ell-2)$ such vertices $w$ and only $\ell-1$ colors availabe, so by the Pigeonhole Principle, the last sentence of the statement holds in this case. A similar argument holds by setting $v_0\ne 0$ and taking any permutation of the remaining entries of $v$.
\end{proof}

\begin{example}\label{tax}
The Desargues graph $ST_k^\ell=ST_2^3$ is 2-choosable according to Theorem~\ref{choo} and the lists $L(v)$ are deduced on the left of Figure~\ref{fig2} from the colored (not light-gray) positions of the tuples representing the corresponding vertices $v$. Any list coloring associated to this is proper, as for example the one shown on the right of the figure. 
However, there is no $(\ell-1)$-list coloring containing only efficient color classes, as verifiable on the figure, where on the right side there are pairs of vertices at distance 2 with a common color.
\end{example} 

\section{Total coloring of $\ell$-set star transposition graphs}\label{nuevo}

\begin{theorem}
Let $k>1$ and let $\ell>1$. Then, the graph $ST_k^\ell$ has a total coloring.
\end{theorem}

\begin{proof} For $\ell>2$,
the proper edge coloring of Theorem~\ref{pec} and any list selection of a proper vertex coloring as in Theorem~\ref{choo} constitute a total coloring of $ST_k^\ell$, since the positions of the entries used for any vertex $v$ and its neighbors are pairwise different, and the positions of the entries used for $v$ and its incident edges are also pairwise different. 

The case $\ell=2$ only differs from that in the previous paragraph in that the list $L(v)$ of each vertex $v$ of $ST_k^\ell=ST_k^2$ has just one color, so choosability is not properly applicable. Such only color differs from the colors of its adjacent vertices, that have pairwise different colors, and differs from the colors of its incident edges, that also have pairwise different colors, so a total coloring also exists in this case.
\end{proof}

\begin{theorem}\label{teofig2}
Let $k>1$. Then, the graph $ST_k^2$ has an efficient coloring via the $\Sigma$E$\,^2$-sets $\Sigma_i^k$ mentioned in item (3) in Section~\ref{s2}.
\end{theorem}

\begin{proof}
The Pigeonhole Principle, as applied in the proof of Theorem~\ref{choo}, cannot be applied properly in this case $\ell=2$, where the SE$^\ell$-sets $S_i^k$ are used however yielding an efficient coloring. But employing the sets $\Sigma_i^k$, we find that there are two {\it types} of colored 6-cycles in a total coloring of $ST_k^2$, illustrated in Figure~\ref{fig3} (see Example~\ref{exfig2} below), namely
\begin{enumerate}
\item 6-cycles, denoted $C(c_1,c_2,c_3)$, having three pairs of opposite edges, where each pair has its two edges with the same color, a total of three colors $c_1,c_2,c_3$;
\item 6-cycles, denoted $C(d_1,d_2)$, having two triples of independent edge of alternate colors $d_1,d_2$.
\end{enumerate}
For any 6-cycle $C=C(c_1,c_2,c_3)$ of type 1 and each color $d\notin\{c_1,c_2,c_3\}$, the six edges of color $d$ with just one endvertex in $C$ have their remaining six endvertices at pairwise distance 3 between them. This is the only way to get vertices of a common color in $ST_k^2$, at pairwise distance 3, and such vertices form part of a $\Sigma$E-set of $ST_k^\ell$ by means of the same (continued) argument presented so far. In fact, those six edges of color $d$ determine six triples of independent color-$d$ edges in six corresponding 6-cycles $C'$ of type 2 adjacent to $C$. In each such $C'$, the remaining triple of independent edges separates $C'$ not only from $C$ but from two other 6-cycles of type 1, fact useful in continuing the formation of the $\Sigma$E-sets.    
\end{proof}

\begin{example} \label{exfig2}
Figure~\ref{fig3} represents two embeddings of a 6-cycle $C=C(2,3,4)$ of type 1 (forming the inner border of an annular gray area, darker on the left, lighter on the right) in $ST_3^2$ surrounded by two different sextuples of 6-cycles of type 2, where colors are numbered 1 = red, 2 = blue, 3 = green, 4 = hazel and 5 = black. On the left, (resp. right), representation of the figure, the six (underlined) vertices at distance 1 from $C$ (on the external border of the annular gray area) have color 5 = black, (resp. 1 = red), realized by edges of color 1 = red, (resp. 5 = black).  Clearly, the distance between the underlined vertices is 3.
\end{example}

\begin{corollary}\label{cutout}
In $ST_k^2$ for fixed $k>2$, the union of all 6-cycles of the form $C(d_1,d_2)$, $C(d_1,d_3)$, $C(d_1,d_4)$, $C(d_1,d_5)$ of type 2 with independent edge triples sharing common color $d_1$, where the colors $d_i$ are pairwise different ,for $i=1,2,3,4,5$, form toroidal subgraphs $T_{d_1}(d_2,d_3,d_4,d_5)$ containing the disjoint union of cycles $C(c,c',c'')$ of type 1, ($\{c,c',c'')\subset\{d_2,d_3,d_4,d_5\}$),  with minimum distance 1 produced by edges of color $d_1$. Moreover, from each such $C(c,c',c'')$ depart six edges having the left color $d\in\{d_2,d_3,d_4,d_5\}\setminus\{c,c',c''\}$ with the remaining endvertices belonging to the E-set $\Sigma_{d_1}^k$ of vertices with color $d_1$.   
\end{corollary}
\begin{figure}[htp]
\includegraphics[scale=0.57]{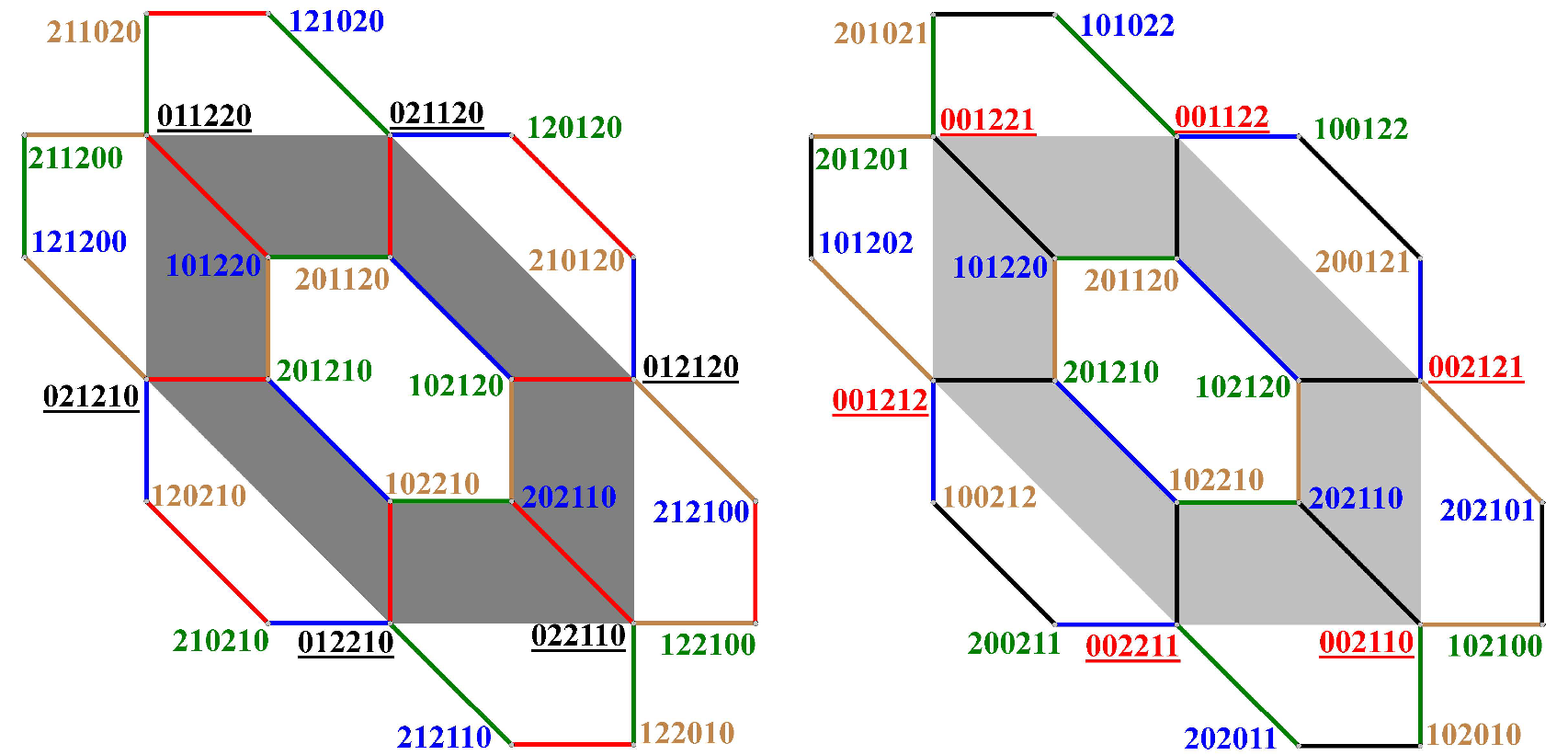}
\caption{Example for the proof of Theorem~\ref{teofig2} and Example~\ref{exfig2}.}
\label{fig3}
\end{figure} 

\begin{proof} The proof is shown here for $k=3$ and $d_1=5$. The general case works likewise. The E-set $\Sigma_{d_1}^k$ is given as the set of vertices $v=d_1v_1v_2v_3v_4d_1$ of degree 1 in the toroidal subgraph $T_5(1,2,3,4)$ obtained from the rectangular cutout shown in Figure~\ref{fig4} by identifying the top and bottom borders, with the bottom border displaced half its length to the right with respect to the top border, (as shown by the displacement between the two parts of the light-gray area, redrawn here from the one in Figure~\ref{fig3}, as is the dark-gray area, too) and identifying in parallel the left and right borders. 
Inside each $C(c,c',c'')$, a pair $d_1v_i$ is highlighted, where $v_i$ must be transposed with the initial $d_1$ in order to transform each vertex $v$ into its adjacent vertex in the 6-cycle $C(c,c',c'')$, ($d_1\in\{0,1,2\}$, $v_i\in\{1,2,3,4\}$). Only a few (underlined) ''black" vertices $v=d_1v_1v_2v_3v_4d_1$ are specified in the figure, but all the otherwise colored vertices are.
\end{proof}

\section{Efficient coloring of 2-set star transposition graphs}
\label{prelim}

In Theorem~\ref{teofig2}, it was shown that the graphs $G=ST^2_k$, ($0<k\in\mathbb{Z}$), 
satisfy the conditions of the following theorem, so that it can be applied in Section~\ref{s3}.

\begin{figure}[htp]
\includegraphics[scale=1.65]{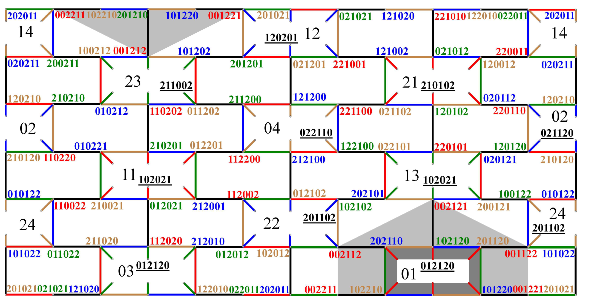}
\caption{Cutout of toroidal subgraph in Corollary~\ref{cutout} in relation to Example~\ref{exfig2}.}
\label{fig4}
\end{figure} 

Let $3<h\in 2\mathbb{Z}$. Let $G=(V(G),E(G))$ be a connected $(h-2)$-regular graph with an efficient coloring using up the color set $[h]\setminus\{0\}=\{1,\ldots,h-1\}$. 
We use the inequality $\chi''(G)\ge\Delta(G)+1$, where $\Delta(G)$ is the maximum degree of $G$ \cite{tc-as}. In our case, 
an efficient coloring provides a partition of $V(G)$ into $h-1$ subsets $W_1,\ldots,W_{h-1}$, where $W_i$ is formed by those vertices of $G$ having color $i$, for each $i\in[h]\setminus\{0\}$. Moreover, each $W_i$ is an E-set of $G$, for $i\in[h]\setminus\{0\}$. 
In Theorem~\ref{chi}, a family of disjoint subgraphs of $G$ may contain (degenaratively) just one subgraph, in which case the union of such family is still said to be a disjoint union, as is the case of Example~\ref{joder!}, below. 

\begin{theorem}\label{chi}  Let $4<h\in 2\mathbb{Z}$ and let $i\in[h]\setminus\{0\}$.
Let $G=(V(G),E(G))$ be a connected $(h-2)$-regular graph with an efficient coloring using up the color set $[h]\setminus\{0\}=\{1,\ldots,h-1\}$. Such a coloring provides a partition of $V(G)$ into $h-1$ classes $W_1,\ldots,W_{h-1}$, where $W_i$ is formed by those vertices of $G$ having color $i$, for each $i\in[h]\setminus\{0\}$. Then each $G\setminus W_i$ is a connected $(h-3)$-regular subgraph such that at most $h-1$ colors are required for it to have an efficient coloring.  
Letting $E_i$ be the set of edges with color $i$ in $G\setminus W_i$, we have that:\begin{enumerate} \item $G\setminus W_i\setminus E_i$ is the disjoint union of copies of regular subgraphs of degree $h-4$ having corresponding efficient colorings with $h-3$ colors, each copy obtained from $[h]\setminus\{0,i\}$ by removing the edges of a specific color $j\ne i$;
\item $G\setminus E_i$ is a non-bipartite $(h-2,h-3)$-biregular graph.
\end{enumerate} 
\end{theorem}

\begin{example}\label{joder!}
The complete graph $G=K_5$ with vertex set $V(G)=[5]=\{0,1,2,3,4\}$ has an efficient coloring with each vertex referred to by its own color in $[5]$ and such that the colors for the edges $01,12, 23, 34, 40, 02, 13, 24, 30, 41$ are $3,4,0,1,2,1,2,3,4,0$, respectively. Then, $G\setminus W_0\setminus E_0=G\setminus\{0\}\setminus\{23,41\}$ is the disjoint union of a sole component with its inherited (or restricted) coloring in color set $[5]\setminus\{0\}=\{1,2,3,4\}$ still being efficient. This is a case that justifies the initial assumption that $3<h\in 2\mathbb{Z}$, previous to the theorem.      
\end{example}

\begin{proof}  By definition of efficient coloring, each $W_i$ is an E-set. Deleting $W_i$ from $G$ removes all the edges incident to the vertices of $W_i$, so $G\setminus W_i$ has an efficient coloring, which is not efficient since there is an edge color lacking incidence to each particular vertex of $G\setminus W_i$.
To establish item 1, notice that removal of $E_i$ from $G\setminus W_i$ for $h>4$, leaves us with the graph induced by the edges of all colors other than color $i$, which necessarily disconnects $G\setminus W_i$, again because of the definition of efficient coloring. 
To  establish item 2, the removal of the edges with color $i$ leaves their endvertices with degree $h-3$ and forming a vertex subset of the resulting $G\setminus E_i$, while the remaining vertices have color $i$, degree $h-2$ and form a stable vertex set.
This completes the proof of the theorem.
\end{proof}

\begin{example}
The odd complete graphs $K_{2n+1}$ ($n\ge 1$) satisfy the conditions of Theorem~\ref{chi} with $h=n+1$, but does not satisfy our initial assumption that the girth be larger than 3. In fact, a total coloring \cite{Yap} of $K_{2n+1}$ is given by numbering the vertices $0, 1, \ldots, 2n$ so that the edges $(j-i)(j+i)$ with $j\in\{0,\ldots,2n\}$ get color $j$ while each vertex $j$ gets precisely color $j$, where $0<i\le n$ and $j\pm i$ is taken mod $2n+1$. This total coloring is efficient, with each vertex color class formed by just one vertex. 
\end{example}

\section{$\Sigma$E-sets of star 2-set transposition graphs}\label{s3}

Let $k>1$. As said at the end of Section~\ref{s2}, 
we consider here only the graphs $ST^2_\ell$ and their vertex subsets $\Sigma_i^k$ formed by the vertices $v_0v_1\cdots v_{2kl-1}$ such that $v_0=v_i$ for exactly one value $i\in\{1,\ldots,2k-1\}$.
We notice that the graph $ST^2_k$ has $\frac{(2k)!}{2^k}$ vertices and regular degree $2k-2$.

\begin{theorem}\label{ultimo}
Let $i\in[2k]\setminus\{0\}=\{1,\cdots,2k-1\}$ and let $\Sigma_i^k$ be the set of vertices $v_0v_1\cdots v_{2k-1}$ of $ST^2_k$ such that $v_0=v_i$, where $i=1,\ldots,2k-1$. 
Then, the $2k-1$ sets $\Sigma_i^k$ are $\Sigma$E-sets forming a partition of $V_k^2=V(ST^2_k)$. 
\end{theorem}

\begin{proof}
The sets $\Sigma_i^k$ are $\Sigma$E-sets, since every vertex of $V_k^2\setminus\Sigma_i^k$ is a neighbor of just one vertex of $\Sigma_i^k$, for each $i\in[2k]\setminus\{0\}$.
 In fact, every vertex $v\in V_k^\ell\setminus\Sigma_i^k$ is either of the form $v_0\cdots v_i\cdots v_j\cdots v_{2k-1}$ or $v_0\cdots v_j\cdots v_i\cdots v_{2k-1}$ with $v_i=v_j$.
 Applying the transposition $(0\;j)$ to $v$ yields a neighbor of $v$ in $\Sigma_i^k$.
 Moreover, the sets $\Sigma_i^k$  form a partition of $V_k^2$, since each $\Sigma_i^k$ is the set of vertices $v_0v_1\cdots v_{k\ell-1}$ of $ST^2_k$ with $v_0=v_i$, for $i=1,\ldots,2k-1$.
\end{proof}

\begin{theorem}\label{t2}
Let $k>2$ and let $i\in[2k]\setminus\{0\}$. Let $E_i^k$ be the set of edges having color $i$ in $G\setminus\Sigma_i^k$, provided by Theorem~\ref{pec}. Then, no edge of $E_i^k$ is incident to any vertex of $\Sigma_i^k$. 
\end{theorem}

\begin{proof}  Corollary~\ref{cutout} and Figure~\ref{fig4} offer a case of the statement for $d_1=v_{2k-1}$, as all the (black) edges $(v,w)$ of color $2k-1$ are obtained by permuting the first and last entries of $v$ and $w$, while a vertex $u$ of color $2k-1$ has the first and last entries with the same color, preventing $u$ to be the endvertex of color $2k-1$.
\end{proof}

\begin{example}\label{abrazo} In $ST_2^2$,
$\Sigma_1^2$, depicted in the lower left of Figure~\ref{fig1} as the 6-cycle induced by the external thick (not thin) edges, has an efficient coloring with $\Sigma_1^2=\{0011,1100\}$ in color blue, as is also $E_1^2=\{(0101,1001),(0110,1010)\}$; $\Sigma_2^2=\{0101,1010\}$ in color green, as is also $E_2^2=\{(0110,1100),(0011,1001)\}$; $\Sigma_3^3=\{0110,1001\}$ in color red, as is also $E_3^2=\{(0011,1010),(0101,1100)\}$. (The thin edges appear in Example~\ref{sigue}).
\end{example}

\begin{theorem}\label{illus} The 
partitions into $\Sigma$E-sets $\Sigma_i^k$
of the graphs $ST^2_k$, for $i\in[2k]\setminus\{0\}$, extend to partitions into $\Sigma$E-sets of supergraphs of the graphs $G=ST^2_k$, but such partitions cannot be completed into efficient colorings.
\end{theorem}

\begin{proof} Let $G=ST^2_k$ and let $h=2k$.
In the setting of Theorem~\ref{chi}, let $W_i=\Sigma_i^k$, for $i\in[h]\setminus\{0\}$. 
Let us number the vertices of $W_i$, for $i\in[h]\setminus\{0\}$, by writing $W_i=\{w_i^1,\ldots,w_i^t\}$, where $t=|W_i|$ is the common cardinality of the classes $W_i$. Let us add to $G$ vertices $i=1,\ldots,t$, with the new vertex $i$ made adjacent to each vertex in $\{w_i^1,\ldots,w_i^t\}$. We add a new color $h\notin[h]\setminus\{0\}$ for the new vertices $1,\ldots,t$. This way, the resulting graph has a partition into $\Sigma$E-sets, but not an efficient coloring. In fact, the new edges adjacent to the new vertex $i\in\{1,\ldots,t\}$ should have pairwise different colors for the resulting coloring to be total, which is not the case since all colors in $\{2,\ldots,h-1,h\}$ were employed already, a contradiction.
\end{proof}

\begin{example}\label{sigue}
 To illustrate Theorem~\ref{illus} , the graph $ST_2^2$ depicted on the lower left of Figure~\ref{fig1} for Example~\ref{abrazo} has two additional central vertices 1 and 2 and six new thin edges from $W_1=\{0011,0101,0011\}$ onto vertex 1 and from $W_2=\{1100,1010,1001,1100\}$ onto vertex 2, leading to complete the 6-cycle $ST_2^2$ into an image of the 3-cube and exemplifying the observed contradiction.
\end{example}

\begin{remark}\label{occur} The total coloring of $ST^2_k$ will be referred to as its {\it color structure}. 
The $k2^{k-1}$ copies of $ST^2_{k-1}$ in $ST^2_k$ whose disjoint union is $ST^2_k\setminus\Sigma_i^k\setminus E_i^k$ inherit each a color structure that generalizes that of the 3-colored 6-cycles in $ST^2_3\setminus\Sigma_5^3$ and is similar to the color structure of $ST^2_{k-1}$.
\end{remark}

\begin{theorem}\label{app} The graphs $ST^2_k$ satisfy the conditions of Theorem~\ref{chi}, so they also satisfy its conclusions. 
\end{theorem}

\begin{proof} Let us see that the hypotheses of Theorem~\ref{chi} are satisfied by taking $h=2k$, $G=ST^2_k$, $W_i=\Sigma_i^k$ and $E_i=E_i^k$, where $i\in[h]\setminus\{0\}$.
Then, it is seen that $ST^2_k\setminus W_i$ is a connected $(h-3)$-regular subgraph such that at most $h-1$ colors are required for it to have an efficient coloring.
Moreover, Theorems~\ref{ultimo} and~\ref{t2} imply item 1. Furthermore, item 2 holds, since clearly
$ST^2_k\setminus E_i$ is a non-bipartite $(h-2,h-3)$-biregular graph, as the vertices of $W_i$ have degree $h-2$ in $ST^2_k$ while the remaining vertices, having each lost the adjacency of an edge of $E_i$, have degree $h-3$.
\end{proof}

\section{SE- and $\Sigma$E-chains of 2-set star transposition graphs}\label{Cs}
 
In \cite{D73}, a countable family of graphs ${\mathcal G}=\{\Gamma_1\subset\Gamma_2\subset\cdots\subset\Gamma_i\subset\Gamma_{i+1}\subset\cdots\}$ was said to be an {\it E-chain} if every $\Gamma_i$ was an induced subgraph of $\Gamma_{i+1}$ and each $\Gamma_i$  had an E-set $C_i$, for $0<i\in\mathbb{Z}$.
Let $\kappa_i$ be an {\it inclusive} map of $\Gamma_i$ into $\Gamma_{i+1}$, meaning that $\kappa_i(\Gamma_i)$ is an induced subgraph of $\Gamma_{i+1}$, for $i\ge 1$. 
If $C_{i+1}$ was the open neighborhood \cite{D73} $N(\kappa_i(V(\Gamma_i)))$ of $\kappa_i(V(\Gamma_i))$, then the E-chain $\mathcal{G}$ was said to be a {\it neighborly} E-chain.
A particular case of E-chain $\mathcal{G}$ has each $C_{i+1}$ split as a partition of images $\zeta_i^{(j)} (C_i)$ of $C_i$ through respective inclusive maps $\zeta_i^{(j)}$, where $j$ varies on some  finite indexing set. In such a case, the E-chain was said to be {\it segmental} \cite{D73}. 
An E-chain $\mathcal{G}$ was said to be {\it dense} \cite{D73}, if for each $n\ge 1$ it is $|V(\Gamma_n)|=(n+1)!$ and $|C_n|=n!$ 

In \cite{D73}, a dense segmental neighborly E-chain
$${\mathcal{ST}}(1)=\{ST^1_1\subset ST^1_2\subset\cdots\subset ST^1_k\subset ST^1_{k+1}\subset\cdots\}$$
was constructed, but such an E-chain is not suitable in the context of the graphs $ST^2_k$ and say their $\Sigma$E-sets $\Sigma_i^k$. Instead, 
the graphs $ST^2_k$ form a $\Sigma$E-chain
\begin{eqnarray}\label{G2}{\mathcal{ST}}(2)=\{ST^2_1\subset ST^2_2\subset\cdots\subset ST^2_k\subset ST^2_{k+1}\subset\cdots\},\end{eqnarray}
with the inclusions $ST^2_k\subset ST^2_{k+1}$ realized by a set of $k+1$ maps
\begin{eqnarray}\label{z1}\kappa_k^j:ST^2_k\rightarrow ST^2_{k+1},\;\mbox{ where }j\in[k+1],\end{eqnarray}  that are {\it neighboring}, meaning that: {\bf(a)} the images $\kappa_k^j(ST^2_i)$ are pairwise disjoint, and {\bf(b)}
\begin{eqnarray}\label{z2}\Sigma_k^{k+1}=\cup_{j=1}^{k-1}N(\kappa_i^j(V^2_i))\end{eqnarray} is a disjoint union such that \begin{eqnarray}\label{z3}\kappa_k^j(v_0v_1\cdots v_{2k-2}v_{2k-1})=v^j_0v^j_1\cdots v^j_{2k-2}v^j_{2k-1}jj,\;\mbox{ for }j\in[k+1],\end{eqnarray}
where 
\begin{eqnarray}\label{z4}v^k_i=v_i,\; v^{k+1}_i=v_i+1\!\mod(k+1),\; \ldots, v^{k+h}_i=a_i+h\!\mod(k+1), \ldots,\end{eqnarray} for $i\in[2k]$, with the superindices $k+h$ of the entries $a^{k+h}_j$ taken mod $(k+1)$, too.

A $\Sigma$E-chain as in display (\ref{G2}) where each inclusion $ST^2_k\subset ST^2_{k+1}$ is realized by $k+1$ neighboring maps $\kappa_k^j$, as in displays (\ref{z1}) to (\ref{z4}), is said to be a
{\it disjoint neighboring} $\Sigma$E-chain. 

\begin{example}
In Figure~\ref{fig4}, the three 6-cycles $d_1v_i=24, 04, 14$ (notation specified in the proof of Theorem~\ref{cutout}) are the images of the graph $ST_2^2$, consisting of the 6-cycle $$(0011,1001,0101,1100,0110,1010),$$ by
adding the suffix $jj$ to the modified 6 cycle obtained by adding 0, 1, 2 mod 3 respectively to the entries of such 6-cycle.   
\end{example}

The notion of segmental E-chain can also be generalized to the case of the graphs $ST^2_k$, where we replace ``neighborly" by ``disjoint neighboring". In such case, the $\Sigma$E-chain will be said to be {\it disjoint fragmental}. It is clear by symmetry that the $\Sigma$E-chain ${\mathcal{ST}}(2)$ of display (\ref{G2}) is disjoint fragmental. 

\begin{remark}\label{Sch}
The density of ${\mathcal ST}(1)$ in \cite{D73} does not help for the $\Sigma$E-chain ${\mathcal ST}(2)$. 
Even though for $k>1$ the graph $ST^1_k$ is the Cayley graph of $Sym_k$ generated by the transpositions $(0\;i)$, ($0<i<k$), we have that for $\ell>1$ the graph
$ST^\ell_k$ is neither a Cayley graph nor a Schreier coset graph of the quotient of $Sym_{k\ell}$ modulo say its subgroup $H_\ell$ generated by the transpositions $(i\;i+1)$, ($0\le i<k$), because the edges of $ST^\ell_k$ are not given by transpositions $(0\;i)$ independently of the values $i$ in different vertices of $ST^\ell_k$. We recur to the following. 
\end{remark}

Given a group $G$, a subgroup $H$ of $G$ and a generating set $S(Hg)$ for each right coset $Hg$ of $H$ in $G$, a {\it Schreier local coset graph} of $G$ is a graph whose vertices are the right cosets $Hg$ and whose edges are of the form $(Hg,Hgs)$, for $g\in G$ and $s\in S(Hg)$.  
 The density of \cite{D73} must be replaced by this in order to be useful in the context of ${\mathcal ST}(2)$. The $\Sigma$E-sets found in the graphs $ST^2_k$ become
 as dense as they can be via such modification, so we say that such $\Sigma$E-sets are {\it thick}. We have the following.

\begin{theorem} 
The $\Sigma$E-chain ${\mathcal{ST}}(2)$  of display (\ref{G2}) is  a thick, disjoint fragmental, disjoint neighboring $\Sigma$E-chain with the $\Sigma$E-sets $\Sigma_i^k$ of Theorem~\ref{t2}.
\end{theorem}

\begin{proof} 
Displays like (\ref{esq}), but for $k>2$, extend the definition of a Schreier coset graph as  said above. The example in display~\ref{esq} shows that $ST^2_2$ is a Schreier local coset graph of the group $V(ST^1_4)$, its subgroup $H$ generated by the transpositions $(0\;1)$ and $(2\;3)$, and the local generators indicated in the last line of the display. This shows downward from the top the right cosets of $V(ST^1_4)$ mod the subgroup generated by the transpositions $(0\;1),(2\;3)$, and then the representations of such right cosets as vertices of $ST^2_2$.  
\begin{eqnarray}\label{esq}
\begin{array}{l||c|c|c|c|c|c}
Right                &0123 & 2301 &0213 &2031 &0231&2013\\
cosets\; of        &0132 &2310  &0312 &2130 &0321&2103\\
V(ST^1_4)         &1023 &3201  &1203 &3021 &1230&3012\\
mod\;H             &1032  &3210 &1302 &3120 &1320&3102 \\\hline
V(ST^2_2)         &0011 &1100  &0101 &1010 &0110&1001\\\hline
Gen.\; set&(0\;2),(0\;3)&(0\;2),(0\;3)&(0\;1),(0\;3)&(0\;1),(0\;3)&(0\;1),(0\;2)&(0\;1),(0\;2)\\
\end{array}
\end{eqnarray}
In a similar way, it is seen for $k>2$ that $ST^2_k$ is a Schreier local coset graph of $V(ST^2_k)$ with respect to the subgroup generated by the transpositions $(2a\;2a+1)$ with $0\le a<k$. The discussion and remaining modifications above in this section provide the rest of the properties in the statement.
\end{proof}

\section{From 2-set star to pancake permutation graphs}\label{panqueque}

Let $0<k,\ell\in\mathbb{Z}$. The {\it $\ell$-set pancake permutation graph} $PC_k^\ell$ is defined as follows. First, let $V(PC_k^\ell)=V(ST_k^\ell)=V^\ell_k$ be the set of all $\ell$-set permutations of length $k\ell$ and let $\pi_1=\pi_2$ be the  identity permutation of $V_k^\ell$. For $i>2$, let  $\pi_i$ be an any product of independent transpositions of the set $\{1,\ldots,i-1\}$. Let $A(\pi_1,\ldots,\pi_i,\ldots,\pi_{2k-1})=\{(0\;1)\pi_1,\ldots,(0\;i)\pi_i,\ldots,(0\;2k-1)\pi_{2k-1}\}.$ In \cite[Lemma 2]{D73} it is implied that for any choice of the permutations $\pi_i$, for $i\ge 3$, the set $A(\pi_1,\ldots,\pi_{2k-1})$ generates $Sym_{2k-1}$, so the sequence of Cayley graphs with generating set $A(\pi_1,\ldots,\pi_{2k-1})$ forms a chain of nested graphs ${\mathcal G}=\{\Gamma_1\subset\Gamma_2\subset\cdots\subset\Gamma_k\subset\Gamma_{k+1}\subset\cdots\}$
 with natural inclusions $\Gamma_k\subset\Gamma_{k+1}$.

If we choose $\pi_1=\pi_2=\cdots=\pi_{2k-1}$ in $A(\pi_1,\ldots,\pi_{2k-1})$, we get the $\ell$-set star transposition graphs $ST^\ell_k$ and a resulting SE$^\ell$-chain ${\mathcal G}={\mathcal ST}^\ell$ with $\Gamma_k=ST^\ell_k$. If we choose
$\pi_i=(1\;i-1)\cdots(\lfloor i/2\rfloor\;\lceil i/2\rceil)$, for $i=3,\ldots,k-1$, we get the {\it $\ell$-set pancake permutation graph} $PC_k^\ell$ and a resulting SE$^\ell$-chain ${\mathcal G}={\mathcal PC}^\ell$. 
In particular, the $\ell$-set pancake permutation graph $PC^\ell_k$ has the same vertex set as $ST^\ell_k$ and its edges involve each the maximal product of concentric disjoint transpositions in any prefix of an endvertex string, including the external transposition (representing an edge of $ST^\ell_k$).
The graphs $PC_k^1$ were seen in \cite{D73} to form a dense segmental neighborly E-chain 
$\mathcal{PC}(1)=\{PC_1^1,PC_2^1,\ldots,PC_k^1,\ldots\}$. (Figure 2 of \cite{D73} represents the graph $PC^1_4$). In a similar fashion to that of Section~\ref{Cs}, the following partial extension of that result can be established.
 
\begin{theorem} 
The chain $\mathcal{PC}(2)=\{PC_1^2,PC_2^2,\ldots,PC_k^2,\ldots\}$ is  a thick, disjoint neighborly $\Sigma$E-chain with the $\Sigma$E-sets $\Sigma_i^k$ of Theorem~\ref{t2}, but it fails to be disjoint segmental. A similar result is obtained for any choice of the involutions $\pi_1,\pi_2,\ldots,\pi_i\ldots$ with not all the $\pi_i$s being identity permutations.
\end{theorem}

\begin{proof} Adapting the arguments given for ${\mathcal ST}(2)$ in Subsection~\ref{Cs} can only be done for the $\Sigma$E-sets $\Sigma_i^k$ in 2-set pancake permutation graphs, since the feasibility for the sets $\Sigma_i^k$, ($1\le i<2k-1$),  to be $\Sigma$E-sets is obstructed by the pancake permutations in $A(\pi_1,\ldots,\pi_{2k-1})$, meaning that we can only establish that the $\Sigma$E-chain $\mathcal{PC}(2)$ is dense and disjoint neighborly, but not disjoint segmental. The ``black'" vertices, those whose color is $2k-1$, form an $\Sigma$E-set $\Sigma^k_i$ with the desired properties, and their removal leaves a $2k-2$-regular graph from which the removal of the ``black" edges, forming an edge subset $E^2_i$, leaves the disjoint union of the open neighborhoods $N(v)$ of the vertices $v$ in the $\Sigma$E-set $\Sigma^k_i$. This behavior is similar for any other choice of the involutions $\pi_1,\pi_2,\ldots,\pi_i\ldots$ with not all the $\pi_i$s being identity permutations, other than $\pi_i=(1\;(i-1))\cdots(\lfloor i/2\rfloor\;\lceil i/2\rceil)$, for $i=3,\ldots,k-1$, which were used precisely to define the pancake permutation graphs.
\end{proof}


\begin{thebibliography}{99}

\bibitem{AK} S. Arumugam and R. Kala, {\it Domination parameters of star graph}, Ars Combinatoria,  {\bf 44} (1996), 93--96.

\bibitem{BBS} D. W. Bange, A. E. Barkausas and P. J. Slater, {\it Efficient dominating sets in graphs}, in: R. D. Ringeisen and F. S. Roberts, eds., Applications of Discrete Math. (SIAM, Philadelphia, 1988) 189--199.

\bibitem{gen} D. W. Bange, A. E. Barkausas, L. H.  Host and P. J. Slater, {\it Generalized domination and
efficient domination in graphs}, Discrete Math., {\bf 159} (1996), 1--11

\bibitem{B1} M. Behzad, {\it Graphs and their chromatic numbers}, PhD thesis, Michigan State University, 1965.

\bibitem{B2} M. Behzad, {\it The total chromatic number}, Proc. Conf. Combin. Math. and Appl. (1969), 1--8. 

\bibitem{Borges} J. Borges and J. Rif\'a, {\it A characterization of 1-perfect additive codes}, IEEE Transactions on Information Theory, {\bf 46} (1999), 1688--1697. 

\bibitem{Mazzu} S. Dantas, C. M. H. de Figueiredo, G. Mazzuocollo, M. Preissmann, V. F. dos Santos, D. Sasaki, {\it On the total coloring of generalized Petersen graphs}, Discrete Math., {\it 339} (2016), 1471--1475.

\bibitem{prev} I. J. Dejter, {\it Total coloring of regular graphs of girth = degree +1}, Ars Combinatoria, {\bf 162} (2025), 159--176.

\bibitem{cong} I. J. Dejter, {\it On efficient total colorings of regular graphs}, Congressus Numerantium, {\bf 236} (2025), 3--13.

\bibitem{D80} I. J. Dejter, {\it SQS-graphs of extended 1-perfect codes}, Congressus Numerantium {\bf 193} (2008), 175--194.


\bibitem{D73} I. J. Dejter and O. Serra, {\it Efficient dominating sets in Cayley graphs}, Discrete Appl. Math., {\bf 129} (2003), 319--328.

\bibitem{D76} I. J. Dejter and O. Tomaiconza, {\it Nonexistence of efficient dominating sets in the Cayley graphs generated by transposition trees of diameter 3}, Discrete Appl. Math., {\bf 232} (2017), 116--124.

\bibitem{Erdos}, P. Erd\"os, A. L. Rubin and H. Taylor, {\it Choosability in graphs}, Congressus Numerantium, {\bf 26} (1979), 125--157.

\bibitem{Feng} Y. Feng, W. Lin, {\it A concise proof for total coloring subcubic graphs}, Inform. Process. Lett., {\bf 113} (2013), 664--665.

\bibitem{tc-as} J. Geetha, N. Narayanan and K. Somasundaram, {Total coloring-a survey}, AKCE int. Jour. of Graphs and Combin., {\bf 20}, (2023), issue 3. 339-351. 

\bibitem{faltaba} P. Gregor, A. Merino and T. M\"utze, {\it Star transpositions Gray codes for multiset permutations}, J. of Graph Theory, {\bf 103(2)}, (2023), 212--270.

\bibitem{edig} T. W. Haynes, S. T. Hedetniemi and M. A. Henning, {\it Efficient domination in graphs}, in: Domination in graphs:  Core concepts, Springer Monographs in Mathematics, (2023), 259--289.

\bibitem{Huffman} W. G. Huffman and V. Pless, Fundamentals of Error Correcting Codes, Cambridge University Press, 2003.

\bibitem{toft} T. R. Jensen and B. Toft, Graph coloring problems, Wiley-Interscience, New York, 18--21.

\bibitem{Knor} M. Knor and P. Poto\v{c}nik, {\it Efficient domination in vertex-transitive graphs}, Eur. Jour. Combin., {\bf 33} (2012), 1755--1764.

\bibitem{MWS} F. J. MacWilliams and N. Sloane, The Theory of Error-Correcting Codes, North Holland Mathematical Library, Volume 16, 1977.

\bibitem{Vera} V. Pless, Introduction to the Theory of Error-Correcting Codes, John Wiley \& Sons, third edition, 1998.

\bibitem{Rosen} M. Rosenfeld, {\it On the total chromatic number of a graph}, Israel J. Math., {\bf 9} (1971), 396--402.

\bibitem{Arroyo} A. S\'anchez-Arroyo, {\it Determining the total coloring number is NP-hard}, Discrete Math., {\bf 78} (1979), 315--319.

\bibitem{Thomassen} C. Thomassen, {\it Every planar graph is 5-choosable}, J. Combin. Theory, ser. B, {\bf 62} (1994), 180--181.

\bibitem{Vi} N. Vijayaditya, {\it On total chromatic number of a graph}, J. London Math. Soc., {\bf 2} (1971), 405--408.

\bibitem{V} V. G. Vizing, {\it On an estimate of the chromatic class of a $p$-graph}, Discret Analiz, {\bf 3} (1969), 25--30.

\bibitem{V2}, V. G. Vizing, {\it Vertex colorings with given colors}, Metody Diskret Analiz., {\bf 29} (1976), 3--10.

\bibitem{Yap} H.-P. Yap, Total colourings of graphs, Lecture Notes in Mathematics, vol. 1623, Springer/Verlag, Berlin/Heidelberg, 1996.

\end{thebibliography}
\end{document}